\documentclass[12pt]{article}
\usepackage{graphicx}
\usepackage{amsmath,amssymb,amsthm,stmaryrd}
\usepackage[all]{xy}
\xyoption{arc}

\newtheorem{thm}{Theorem}
\newtheorem{cor}[thm]{Corollary}
\newtheorem{prop}[thm]{Proposition}
\newtheorem{lem}[thm]{Lemma}
\theoremstyle{definition}

\theoremstyle{remark}

\newcommand{\mb}[1]{\mathbb{#1}}

\newcommand{\cF}{\overline {\mb F}}

\newcommand{\Mod}[1]{{#1}\mbox{-}{\text{mod}}}
\newcommand{\Alg}[1]{{#1}\mbox{-}{\text{alg}}}

\title{A note on $H_\infty$ structures}
\author{Tyler Lawson\thanks{Partially supported by NSF
    grant 0805833.}}

\begin{document}
\maketitle

\begin{abstract}
We give a source of (coconnective) examples of $H_\infty$ structures that do not lift to $E_\infty$ structures, based on Mandell's proof of the equivalence between certain cochain algebras and spaces.
\end{abstract}

\section{Introduction}

The diagonal map on a space $X$ gives rise to a structure on its cochains which is not commutative, but associative and commutative up to higher coherences (an $E_\infty$ structure).  From this structure, the Steenrod reduced power operations arise; this was axiomatized in \cite{may-steenrodoperations}.  One perspective would regard this as the structure of an $E_\infty$ algebra on the function spectrum $(H\mb F_p)^X$, parametrizing maps from $X$ into an Eilenberg-Mac Lane spectrum.

Power operations and the Adem relations between them do not make use of the entirety of an $E_\infty$ structure.  In the paper \cite{may-hinfty} and in the book \cite{h-infty}, the notion of an $H_\infty$ ring spectrum was developed.  This structure gives enough data to produce geometric power operations, but only depends on structure in the homotopy category.

An $E_\infty$ structure seems to contain strictly more information than an $H_\infty$ structure.  However, decades passed before Noel produced an example of a ring spectrum with an $H_\infty$ structure that does not lift to an $E_3$ structure \cite{noel-hinftyeinfty}, based on a counterexample to the transfer conjecture due to Kraines and Lauda.

In this paper we observe another source of such examples.  In order to give the reader an idea of where we are going, we present the following.

\begin{thm}
\label{thm:main}
Let $R_k$ be a wedge of Eilenberg-Mac Lane spectra such that $\pi_* R_k$ is isomorphic to the graded ring $\mb F_2[x]/(x^3)$, with $|x| = -2^k$.  For $k > 3$, there exists an $H_\infty$ structure on $R_k$ which does not come from an $E_\infty$ structure.
\end{thm}

The $H_\infty$ structures in the theorem can still be defined for $k \leq 3$, but these do come from $E_\infty$ algebras: specifically, the function spectra $(H\mb F_2)^{\mb{RP}^2}$, $(H\mb F_2)^{\mb{HP}^2}$, and $(H\mb F_2)^{\mb{OP}^2}$.

In fact, our construction is essentially equivalent to the Hopf invariant one problem.  We observe that an $H_\infty$ structure on such a spectrum can be specified by describing Dyer-Lashof operations on the homotopy groups, and that these Dyer-Lashof operations are determined by the underlying $H_\infty$ structure.  However, to have a lift to an $E_\infty$ structure, compatibility with the connective cover and results of Mandell \cite{mandell-einftypadic} imply that we must be able to find a space with this cohomology algebra.

Both this paper and Noel's rely on heavy machinery.  (A more direct proof here would lift Adams' proof from secondary Steenrod operations to secondary Dyer-Lashof operations, which would be intrinsically worthwhile.)  Folklore has it that an $H_\infty$ structure does not automatically determine an $A_4$-structure in the sense of Stasheff's associahedra, and so one might suspect that a simpler obstruction exists in terms of Toda brackets.

In addition, it is not clear in either case whether there are situations where there is a ring spectrum which {\em admits} an $H_\infty$ structure, but not a (possibly unrelated) $E_\infty$ structure, because both results rely on the impossibility of lifting a fixed $H_\infty$ structure.

\section{$H_\infty$ structures}

We will work in a highly developed category of spectra where smash-powers are homotopically well-behaved, and $E_\infty$ algebras and commutative ring objects have the same homotopy theory; these include the $S$-modules of \cite{ekmm} and the category of symmetric spectra with the positive stable model structure (the S-model structure of \cite{shipley-convenient}).  We leave the reader to fill in appropriate cofibrancy assumptions.

Let $R$ be a commutative ring object in spectra, with category $\Mod{R}$ of $R$-modules.  There is a monad $\mb P_R$ whose algebras are precisely commutative $R$-algebras:
\[
\mb P_R(M) = \bigvee_{m \geq 0} M^{\wedge_R m} / \Sigma_m
\]
Write $\Alg{R}$ for the category of commutative $R$-algebras.

\begin{prop}
\label{prop:freeforget}
There is a natural isomorphism $T \wedge_R \mb P_R(M) \cong \mb P_T(T \wedge_R M)$.
\end{prop}
\begin{proof}
Both functors are left adjoint to the composite in the commutative diagram of forgetful functors
\[
\xymatrix{
\Alg{T} \ar[d] \ar[r] & \Mod{T} \ar[d] \\
\Alg{R} \ar[r] & \Mod{R}. \\
}\]
Therefore, they are naturally isomorphic.
\end{proof}

The monad $\mb P_R$ is homotopically well-behaved, and descends to a monad on the homotopy category $Ho(\Mod{R})$.  We still denote this
monad by $\mb P_R$, and refer to its algebras as $H_\infty$ $R$-algebras \cite{h-infty}.  If $R$ is the sphere, we simply refer to these as $H_\infty$ algebras.

For $T$ a commutative $R$-algebra, there is a map of monads $\mb P_R \to \mb P_T$.  This represents the forgetful functor from $T$-algebras to $R$-algebras, and on the homotopy category it represents the forgetful functor from $H_\infty$ $T$-algebras to $H_\infty$ $R$-algebras.

We note that this makes the natural isomorphism adjoint to the composite
\[
\mb P_R(M) \to \mb P_R(T \wedge_R M) \to \mb P_T(T \wedge_R M).
\]

\begin{lem}
\label{lem:restrict}
Suppose $T$ is a commutative $R$-algebra and $M$ is a $T$-module which is a retract of $T \wedge_R N$ for some $R$-module $N$.  Then any $H_\infty$ $T$-algebra structure on $M$ is determined by the underlying $H_\infty$ $R$-algebra structure.

In particular, if $M$ is a wedge of suspensions of $T$, this automatically holds.
\end{lem}

\begin{proof}
The free-forgetful adjunction between $T$-modules and $R$-modules allows us to factor the retraction into a sequence of $T$-module maps
\[
M \to T \wedge_R N \to T \wedge_R M \to M,
\]
with composite the identity.  Applying $\mb P_T$ and then making use of Proposition~\ref{prop:freeforget}, we get a retraction
\[
\mb P_T(M) \to T \wedge_R \mb P_R(M) \to \mb P_T(M).
\]
We then calculate homotopy classes of maps, finding that
\[
[\mb P_T(M),M]_{\Mod{T}} \to [\mb P_R(M),M]_{\Mod{R}}
\]
is the inclusion of a retract.  Therefore, any $H_\infty$ structure map $\mb P_T(M) \to M$ is determined by the induced map $\mb P_R(M) \to M$.
\end{proof}

\section{$H_\infty$ algebras over a field}

Fix a prime $p > 0$ and let $H = H\mb F_p$ be an Eilenberg-Mac Lane spectrum, equipped with a commutative ring structure.  We recall the following.

\begin{thm}
The map $\pi_*$, from $Ho(\Mod{H})$ to the category of graded $\mb F_p$-vector spaces, is an equivalence of categories.

Under this equivalence, the monad $\mb P_H$ becomes the monad sending a graded vector space $V$ to the free object on $V$ in the category of graded-commutative algebras over $\mb F_p$ equipped with Dyer-Lashof operations.
\end{thm}
(The statement about the monad follows from \cite[IX.2.2.1]{h-infty}; we have also found the references \cite{rezk-power-operations,lurie-sullivanlectures} beneficial.)

In particular, this makes the category of $H_\infty$ $H$-algebras equivalent to the category of graded-commutative algebras over $\mb F_p$ equipped with Dyer-Lashof operations.

As $H$-modules are all wedges of suspensions of $H$, Lemma~\ref{lem:restrict} implies the following.
\begin{prop}
  For any $H_\infty$ $H$-algebra $A$, the Dyer-Lashof operations on $\pi_* A$ are uniquely determined by the underlying $H_\infty$ structure on $A$.
\end{prop}

We now recall the following consequence of the Adem relations.
\begin{prop}
The functor $A \mapsto \pi_* A$, when restricted to the category of $H_\infty$ $H$-algebras where the operation $P^0$ acts as the identity, is an equivalence to the category of unstable algebras over the Steenrod algebra.
\end{prop}

The standard calculations with the Adem relations in the mod-$2$ Steenrod algebra provide the following.
\begin{prop}
\label{prop:hinftyexists}
A graded ring isomorphic to $\mb F_2[x]/(x^3)$ is the homotopy of an $H_\infty$ $H$-algebra, with the operation $Q^0$ acting as the identity, if and only if $|x|$ is a power of $2$.
\end{prop}

\section{Coconnective $E_\infty$ algebras}

We begin by pointing out the following consequence of the fact that connective covers are compatible with $E_\infty$ algebra structures  \cite{baker-richter-connective}.
\begin{prop}
Suppose that $A$ is an $E_\infty$ algebra which is coconnective, in the sense that $\pi_* A = 0$ for $* > 0$.  Then $A$ canonically admits the structure of an $E_\infty$ $H\pi_0(A)$-algebra.
\end{prop}

\begin{cor}
Suppose that $A$ is a $H_\infty$ $H$-algebra which is coconnective, in the sense that $\pi_* A = 0$ for $* > 0$.  Then the $H_\infty$ structure on $A$ lifts to an $E_\infty$ structure if and only if the $H_\infty$ $H$-algebra structure lifts to an $E_\infty$ $H$-algebra structure.
\end{cor}

\begin{proof}
One direction is clear.  The previous proposition shows that a lifted $E_\infty$ structure automatically comes from an $E_\infty$ $H$-algebra structure, and Lemma~\ref{lem:restrict} then implies that the two $H_\infty$ $H$-algebra structures must agree, because they have the same underlying $H_\infty$ structure by assumption.
\end{proof}

Therefore, to prove nonexistence results it suffices to work entirely within $H$-modules, which is equivalent to working with $E_\infty$ algebras in the category of differential graded modules over $\mb F_p$.

\begin{prop}
Suppose that an $E_\infty$ $\mb F_p$-chain complex $A$ satisfies $H^i(A) = 0$ for $i < 0$, $H^0(A) = \mb F_p$, $H^1(A) = 0$, and $H^i(A)$ finite dimensional over $\mb F_p$.  Then $A \otimes \cF_p$ is equivalent, as an $E_\infty$ differential graded algebra over $\cF_p$, to the singular cochain complex of a space.
\end{prop}

\begin{proof}
The chain complex $A \otimes \cF_p$ is an $E_\infty$ algebra, and the homology groups of $A \otimes \cF_p$ are isomorphic to $H^i(A) \otimes \cF_p$ as algebras over the Dyer-Lashof algebra.  As a result, $A \otimes \cF_p$ satisfies the hypotheses of Mandell's Characterization Theorem \cite{mandell-einftypadic}.
\end{proof}

This application of Mandell's theorem provides us a large library of nonexistence results merely by requiring that there exist a space realizing a given ring with its Dyer-Lashof operations.

\begin{prop}
\label{prop:einftyfails}
A graded ring isomorphic to $\mb F_2[x]/(x^3)$, with Dyer-Lashof operations uniquely determined by the requirement that $Q^0$ acts by the identity, can be the homotopy ring of an $E_\infty$ algebra $A$ if and only if $|x|$ is $1$, $2$, $4$, or $8$.
\end{prop}

\begin{proof}
The nonexistence of maps of Hopf invariant one, due to Adams, implies that there are no spaces with cohomology ring $H^*(X;\mb F_2) \cong \mb F_2[x]/(x^3)$ unless $x$ has one of the given degrees.  As $H^*(X; \mb F_2)$ is the fixed subring of $H^*(X;\cF_2)$ under the action of $Q^0$, we then find that the ring $\cF_2[x]/(x^3)$, where $Q^0 x = x$, cannot be the cohomology ring of a space $X$.
\end{proof}

Combining Propositions~\ref{prop:hinftyexists} and \ref{prop:einftyfails} yields Theorem~\ref{thm:main}.

{\bf Acknowledgements.} The author would like to thank Michael Mandell and Justin Noel for discussions related to this note.

\bibliography{../masterbib}

\end{document}